\documentclass[11pt,a4paper]{article}
\usepackage[latin1]{inputenc}
\usepackage{amsmath,amsthm,amsfonts,amssymb,amscd,latexsym}
\usepackage{enumerate}
\usepackage{epsfig,psfrag}
\usepackage{graphics,graphicx,bezier,float,color,hyperref}
\usepackage{url}
\usepackage{multienum}
\usepackage[table]{xcolor}
\usepackage{multicol,multirow}
\usepackage{fancyvrb}
\usepackage{tcolorbox}
\usepackage{parskip}
\usepackage[toc,page]{appendix}
\usepackage{booktabs}
\usepackage[top=2.7cm, bottom=2.7cm, left=1.5cm, right=1.5cm]{geometry}
\usepackage[none]{hyphenat}
\usepackage{blkarray}
\newtheorem{theorem}{Theorem}[section]
\newtheorem{lemma}[theorem]{Lemma}
\newtheorem{proposition}[theorem]{Proposition}
\newtheorem{cor}[theorem]{Corollary}
\newtheorem{remark}[theorem]{Remark}
\newtheorem{definition}[theorem]{Definition}

\numberwithin{equation}{section}
\numberwithin{table}{section}
\numberwithin{figure}{section}

\title{Cullen and Woodall numbers in Padovan and Perrin sequences}
\author{Herbert Batte$^{1,*}$  \and Eric F. Bravo$^{2}$ \and Florian Luca$^{1,3}$ }
\date{}

\begin{document}
\maketitle
\abstract{ Let $\{P_n\}_{n\ge 0}$ and $\{R_n\}_{n\ge 0}$ denote the Padovan and Perrin 	sequences, both satisfying the recurrence $U_{n+3} = U_{n+1} + U_n$, 
but with initial values $P_0 = P_1 = P_2 = 1$ and $R_0 = 3$, 
$R_1 = 0$, $R_2 = 2$, respectively. A \textit{Cullen number} is a positive integer of the form $m\cdot 2^m + 1$ for some integer $m \ge 1$, while a \textit{Woodall number} is a positive integer of the form $m\cdot 2^m - 1$ for some integer $m \ge 1$. In this paper, we determine all Woodall numbers in the Padovan sequence and all Cullen numbers in the Perrin sequence. Specifically, we prove that $1$ and $7$ are the only Woodall numbers in the Padovan sequence, and that $3$ is the only Cullen number in the Perrin sequence.} 

\medskip


\medskip


\medskip

\thanks{$ ^{*} $ corresponding author}

\section{Introduction}\label{intro}

In 1905, Reverend James Cullen \cite{JC} published the first study on the primality of numbers of the form $m \cdot 2^{m}+1$. He realized that, apart from $1\cdot 2^{1}+1=3$ and with the possible exception of $53\cdot 2^{53}+1$, the first hundred of these numbers are composite. A year later, Cunningham \cite{AC} published the finding that $5591$ divides $53\cdot 2^{53}+1$. The first non-trivial prime number of this form was discovered by Robinson \cite{RR} nearly fifty year later, who confirmed Cunningham's suspicion that $141\cdot 2^{141}+1$ was prime. These numbers are now called \textit{Cullen numbers}. We denote them by $C_m$ for every positive integer $m$. A Cullen number that is prime is called a Cullen prime.

In 1917, Cunningham and Woodall \cite{CW}, inspired by Cullen's work, published the first paper on the primality of numbers of the form $m\cdot 2^{m}-1$. They proved that there were only three prime numbers of this form in the range $1\le m\le 29$, namely $2\cdot 2^{2}-1$, $3\cdot 2^{3}-1$, and $6\cdot 2^{6}-1$. These numbers are now called \textit{Woodall numbers}, and primes of this form are called Woodall primes. We denote them by $W_m$ for all $m\in \mathbb{Z}^{+}$.

In the mid-twentieth century, mathematicians wondered how often Cullen and Woodall primes appeared. In 1976, Hooley \cite{CH} showed that the natural density of positive integers $m\le x$ for which $C_m$ is prime is of the order $o(x)$ when $x\rightarrow \infty$. In other words, almost all of Cullen's numbers are composite. Hooley's proof was reworked by Hiromi Suyama to show that it works for any sequence of numbers of the form $m\cdot 2^{m+a}+b$, where $a$ and $b$ are integers, and in particular also for Woodall numbers. To date, sixteen Culllen primes and thirty-four Woodall primes are known. The largest are $C_{6679881}$ and $W_{17016602}$. They were discovered through distributed computing and citizen collaboration via the international PrimeGrid project.

In 2003, Luca and St\u{a}nic\u{a} \cite{12} began the search for Cullen and Woodall numbers in linear recurrence sequences. They proved that there are only a finite number of them in any binary recurrence sequence that satisfies certain conditions. They applied their result to the Fibonacci and Pell sequences. They found that $1$ and $3$ are the only Cullen numbers that are also Fibonacci numbers, and $1$ is the only Woodall number that is also a Fibonacci number; while the only Cullen number and the only Woodall number present in the Pell sequence is $1$. Recently, Alahmadi and Luca \cite{1} proved that there is no Cullen number that is also a repunit number.

Regarding the presence of Cullen numbers and Woodall numbers in recurrent linear sequences of order greater than $2$, we have the recent results of B\'erczes, Pink, and Young \cite{24}, and Bravo and Irmak \cite{BrIrmak}. B\'erczes et al. \cite{24} considered the $k$-generalized Fibonacci sequence, which is a linear recurrent sequence for each fixed order $k\ge 2$. For $k=2$ and $k=3$, this generalization gives the Fibonacci sequence and the Tribonacci sequence, respectively. They proved that $1$ is the unique Cullen number that is also a $k$-generalized Fibonacci number for all $k>2$. Furthermore, they showed that $1$ is a Woodall number that is present in any $k$-generalized Fibonacci sequence for all $k>2$, while $3$ is a Woodall number that is only present in the Tribonacci sequence.

For their part, Bravo and Irmak \cite{BrIrmak} considered a pair of linear recurrent ternary sequences: the Padovan and Perrin sequences. The \textit{Padovan numbers} $\{P_{n}\}_{n\in \mathbb{Z}}$ are defined by the Fibonacci-like recurrence relation
$$
P_{n+3}=P_{n+1}+P_{n}
$$
with initial values $P_{0}=P_{1}=P_{2}=1$. The \textit{Perrin numbers} $\{R_{n}\}_{n\in \mathbb{Z}}$ are defined by the same recurrence relation as the
Padovan numbers but with starting values $R_{0}=3$, $R_{1}=0$, and $R_{2}=2$.

Bravo and Irmak \cite{BrIrmak} determined that $3$, $9$, and $65$ are all the Cullen numbers appearing in the Padovan sequence, and that $7$ is the only Woodall number appearing in the Perrin sequence. Their approach combined a precise analysis of the $2$-adic valuations of 
$\{P_n - 1\}_{n\ge 0}$ and $\{R_n + 1\}_{n \ge 0}$ to obtain absolute bounds, followed by a computational verification.

It is natural then to ask for the complementary results. What are all the Woodall numbers in the Padovan sequence, and what are all the Cullen numbers in the Perrin sequence? In this paper, we answer both questions completely. Our approach parallels that of 
\cite{BrIrmak}, but requires a new and careful analysis of the $2$-adic valuations of $\{P_n + 1\}_{n\ge 0}$ and $\{R_n - 1\}_{n\ge 0}$, which 
are the objects relevant to Woodall numbers in Padovan and Cullen numbers in Perrin, respectively. We prove the following results.

\subsection{Main Results}

\begin{theorem}\label{thm:1}
	The only Woodall numbers in the Padovan sequence are $1$ and $7$.
\end{theorem}

\begin{cor}
The only Woodall prime in the Padovan sequence is $7$.
\end{cor}

\begin{theorem}\label{thm:2}
	The only Cullen number in the Perrin sequence is $3$.
\end{theorem}

\begin{cor}
The only Cullen prime in the Perrin sequence is $3$.
\end{cor}

The proofs of Theorems \ref{thm:1} and \ref{thm:2} combine an application of Baker's theory of linear forms in three logarithms of algebraic numbers and the $2$-adic valuation of $\{P_n + 1\}_{n\ge 0}$ and $\{R_n - 1\}_{n\ge 0}$, followed by a computational verification.

\section{Preliminaries}

\subsection{The Padovan and Perrin sequences}

In this section, we collect some useful formulas and inequalities concerning Padovan and Perrin numbers. We begin with a closed formula that allows us to directly calculate the nth Perrin number.
\begin{lemma}\label{eq:binet_perrin}
For all $n\in \mathbb{Z}$, we have
\begin{equation*}
	R_n = \alpha^n + \beta^n + \gamma^n,
\end{equation*}
where $\alpha,\beta,\gamma$ are the zeros of $\phi(X):=X^{3}-X-1$.
\end{lemma}
\begin{proof}
Since $R_{n+3}-R_{n+1}-R_{n}=0$, the characteristic polynomial of the Perrin sequence is $\phi(X)$. This polynomial has one real root $\alpha$ and two complex conjugate roots $\beta$ and $\gamma$ given by
$$
\alpha=\sqrt[3]{\frac{9+\sqrt{69}}{18}}+\sqrt[3]{\frac{9-\sqrt{69}}{18}},\quad \beta=-\frac{\alpha}{2}+i\frac{\sqrt{3\alpha^{2}-4}}{2},\quad \gamma =\overline{\beta}.
$$
Since the zeros of $\phi(X)$ are distinct, there exist unique constants $c_{\alpha},c_{\beta},c_{\gamma}\in \mathbb{Q}(\alpha,\beta)$ such that  
$$
R_{n}=c_{\alpha}\alpha^{n}+c_{\beta}\beta^{n}+c_{\gamma}\gamma^{n}\quad \text{for all }n\ge 0.
$$
To determine $c_{\alpha}$, $c_{\beta}$, and $c_{\gamma}$, we solve the system of equations:
\begin{align*}
V\begin{pmatrix}c_{\alpha}\\c_{\beta}\\c_{\gamma}\end{pmatrix}=\begin{pmatrix}3\\0\\2\end{pmatrix},\quad \text{where}\quad V:=\begin{pmatrix}1&1&1\\ \alpha &\beta &\gamma \\ \alpha^{2} & \beta^{2} &\gamma^{2}\end{pmatrix}.
\end{align*}
It follows from Vieta's formulas that the zeros $\alpha,\beta,\gamma$ of $\phi(X)$ must satisfy 
$$
\alpha +\beta +\gamma= -\frac{0}{1}=0,\quad \alpha \beta +\beta \gamma +\gamma \alpha=\frac{-1}{1}=-1\quad {\text{\rm and}} \quad \alpha \beta \gamma =-\frac{-1}{1}=1.
$$ 
Therefore, 
$$
\alpha^{2}+\beta^{2}+\gamma^{2}=(\alpha +\beta +\gamma)^{2}-2( \alpha \beta +\beta \gamma + \gamma \alpha)=0^{2}-2(-1)=2.
$$ 
Using Cramer's rule, we get
\begin{align*}
c_{\alpha}=\frac{\begin{vmatrix}3&1&1\\ 0 &\beta &\gamma \\ 2 & \beta^{2} &\gamma^{2}\end{vmatrix}}{|V|}=\frac{\begin{vmatrix}1+1+1&1&1\\ \alpha +\beta +\gamma &\beta &\gamma \\ \alpha^{2}+\beta^{2}+\gamma^{2} & \beta^{2} &\gamma^{2}\end{vmatrix}}{|V|}&=\frac{\begin{vmatrix}1&1&1\\  \alpha &\beta &\gamma \\ \alpha^{2} & \beta^{2} &\gamma^{2}\end{vmatrix}+\begin{vmatrix}1&1&1\\  \beta &\beta &\gamma \\ \beta^{2} & \beta^{2} &\gamma^{2}\end{vmatrix}+\begin{vmatrix}1&1&1\\  \gamma &\beta &\gamma \\ \gamma^{2} & \beta^{2} &\gamma^{2}\end{vmatrix}}{|V|}\\
&=\frac{|V|+0+0}{|V|}=1,
\end{align*}
where, in the penultimate equation, we use the fact that the determinant of a matrix is zero if it has two identical columns. Similarly, we find that $c_{\beta}=c_{\gamma}=1$. Now, we prove by strong induction that
$$
R_{-n}=\alpha^{-n}+\beta^{-n}+\gamma^{-n}\quad \text{for all}\quad n\ge 1.
$$
Since $\alpha,\beta, \gamma$ are zeros of $\phi(X)$, it follows that $\alpha^{-1}$, $\beta^{-1}$, $\gamma^{-1}$ are zeros of 
$$
X^{3}\phi(1/X)=1-X^{2}-X^{3}.
$$ 
By applying Vieta's formulas to $X^{3}\phi(1/X)$ we get that 
$$
\alpha^{-1}+\beta^{-1}+\gamma^{-1}=-1\qquad {\text{\rm and}}\qquad (\alpha \beta)^{-1}+(\beta \gamma)^{-1}+(\gamma \alpha)^{-1}=0.
$$ 
Hence, 
$$
\alpha^{-2}+\beta^{-2}+\gamma^{-2}=(\alpha^{-1}+\beta^{-1}+\gamma^{-1})^{2}-2[(\alpha \beta)^{-1}+(\beta \gamma)^{-1}+(\gamma \alpha)^{-1}]=(-1)^{2}-2(0)=1.
$$ 
Since $R_{-1}=R_{2}-R_{0}=-1$ and $R_{-2}=R_{1}-R_{-1}=1$, we conclude that the statement holds for $n=1$ and $n=2$. To finish the proof of the base case, note that
\begin{align*}
R_{-3}=R_{0}-R_{-2}=3-1&=3-(\alpha^{-2}+\beta^{-2}+\gamma^{-2})=(1-\alpha^{-2})+(1-\beta^{-2})+(1-\gamma^{-2})\\
&=\alpha^{-3}+\beta^{-3}+\gamma^{-3}.
\end{align*}
Suppose that $R_{-k}=\alpha^{-k}+\beta^{-k}+\gamma^{-k}$ for all $1\le k\le n-1$. Then
\begin{align*}
R_{-n}=R_{-n+3}-R_{-n+1}&=R_{-(n-3)}-R_{-(n-1)}\\
&=\alpha^{-(n-3)}+\beta^{-(n-3)}+\gamma^{-(n-3)}-(\alpha^{-(n-1)}+\beta^{-(n-1)}+\gamma^{-(n-1)})\\
&=\alpha^{-n}(\alpha^{3}-\alpha)+\beta^{-n}(\beta^{3}-\beta)+\gamma^{-n}(\gamma^{3}-\gamma)\\
&=\alpha^{-n}+\beta^{-n}+\gamma^{-n}.
\end{align*}
\end{proof}
For the Padovan sequence, a Binet formula is given below.
\begin{lemma}\label{eq:binet_padovan}
For all $n \in \mathbb{Z}$, we have
\begin{equation*}
	P_n = c_{\alpha}\alpha^n + c_{\beta}\beta^n + c_{\gamma}\gamma^n,
\end{equation*}
where $c_{x}:=(x^{2}+x)/(3x^{2}-1)$ for all $x\in \{\alpha,\beta, \gamma\}$.  
\end{lemma}
\begin{proof}
Let $P(X)=\displaystyle\sum_{n=0}^{\infty}P_{n}X^{n}$ be the generating function of the Padovan sequence; then
\begin{align*}
P(X)&=1+X+X^{2}+\displaystyle\sum_{n=3}^{\infty}(P_{n-2}+P_{n-3})X^{n}\\
&=1+X+X^{2}+X^{2}\displaystyle\sum_{n=1}^{\infty}P_{n}X^{n}+X^{3}\displaystyle\sum_{n=0}^{\infty}P_{n}X^{n}\\
&=1+X+X^{2}+X^{2}(P(X)-P_{0})+X^{3}P(X),
\end{align*}
so
\begin{equation*}
P(X)=\frac{1+X}{1-X^{2}-X^{3}}.
\end{equation*}
Since the roots of $1-X^{2}-X^{3}=0$ are distinct, by partial fractions
\begin{equation*}
P(X)=\frac{1+X}{(1-\alpha X)(1-\beta X)(1-\gamma X)}=\frac{c_{\alpha}}{(1-\alpha X)}+\frac{c_{\beta}}{(1-\beta X)}+ \frac{c_{\gamma}}{(1-\gamma X)}.
\end{equation*}
Here,
$$
c_{\alpha}=\frac{\alpha^{2}+\alpha}{3\alpha^{2}-1},\quad c_{\beta}=\frac{\beta^{2}+\beta}{3\beta^{2}-1},\quad c_{\gamma}=\frac{\gamma^{2}+\gamma}{3\gamma^{2}-1}.
$$
In the above we have used that $X^{3}-X-1=(X-\alpha)(X-\beta)(X-\gamma)$ implies that 
$$
3X^{2}-1=(X-\beta)(X-\gamma)+(X-\alpha)(X-\gamma)+(X-\alpha)(X-\beta),
$$ 
and so 
$$
3\alpha^{2}-1=(\alpha -\beta )(\alpha -\gamma),\quad 3\beta^{2}-1=(\beta -\alpha )(\beta -\gamma),\quad {\text{\rm  and}}\quad 3\gamma^{2}-1=(\gamma -\alpha )(\gamma -\beta).
$$ 
Consequently,
\begin{equation*}
P(X)=c_{\alpha}\displaystyle \sum_{n=0}^{\infty}(\alpha X)^{n}+c_{\beta}\displaystyle \sum_{n=0}^{\infty}(\beta X)^{n}+c_{\gamma}\displaystyle \sum_{n=0}^{\infty}(\gamma X)^{n}=\displaystyle \sum_{n=0}^{\infty}(c_{\alpha}\alpha^{n}+c_{\beta}\beta^{n}+c_{\gamma}\gamma^{n})X^{n}.
\end{equation*}
Thus, Binet's formula for the Padovan sequence is
$$
P_{n}=c_{\alpha}\alpha^n + c_{\beta}\beta^n + c_{\gamma}\gamma^n\quad \text{for all}\quad n\ge 0.
$$
This formula can be extended to negative integers using, for example, strong mathematical induction. We omit the details.
\end{proof}

The following growth estimates hold for both sequences.

\begin{lemma}\label{lem:growth_P}
	For all non-negative integer $n$, we have
	\begin{equation*}
		 P_n \ge \alpha^{n-2}.
	\end{equation*}
\end{lemma}

\begin{proof}
We prove this result by strong induction on $n$. If $n=0,1,2$, then $P_{0}=1>0.57>\alpha^{-2}$, $P_{1}=1>0.76>\alpha^{-1}$, and $P_{2}=1\ge 1= \alpha^{0}$, respectively. Suppose that $P_{k}\ge \alpha^{k-2}$ for all $0\le k\le n-1$. Thus, using the recurrence relation of the Padovan sequence and the fact that $\alpha^{-1}$ is a zero of $X^{3}\phi(1/X)$, we obtain
$$
P_{n}=P_{n-2}+P_{n-3}\ge \alpha^{n-4}+\alpha^{n-5}=\alpha^{n-2}(\alpha^{-2}+\alpha^{-3})=\alpha^{n-2}.
$$
\end{proof}

\begin{lemma}\label{lem:growth_R}
	For all integer $n \ge 7$, we have
	\begin{equation*}
		R_{n}-1>\alpha^{n-1}.
	\end{equation*}
\end{lemma}

\begin{proof}
We prove this result by strong induction on $n$. If $n=7,8,9$, then $R_{7}-1=6>5.41>\alpha^{6}$, $R_{8}-1=9>7.16>\alpha^{7}$, and $R_{9}-1=11>9.49>\alpha^{8}$, respectively. Suppose that $R_{k}-1>\alpha^{k-1}$ for all $7\le k\le n-1$. Thus, using the recurrence relation of the Perrin sequence and the fact that $\alpha^{-1}$ is a root of $X^{3}\phi(1/X)=0$, we obtain
$$
R_{n}-1=R_{n-2}+R_{n-3}-1>(\alpha^{n-3}+1)+(\alpha^{n-4}+1)-1=\alpha^{n-1}(\alpha^{-2}+\alpha^{-3})+1>\alpha^{n-1}.
$$
\end{proof}

Next, Lemma \ref{L1} provides a sum formula for the numbers of Padovan and Perrin (see Sokhuma \cite[Proposition 2.2]{KS}).
\begin{lemma}\label{L1}
	For all positive integers $r,s$ we have the relations:
	\begin{align*}
		P_{r+s}&=P_{s-1}P_{r-1}+P_{s-2}P_{r}+P_{s-3}P_{r-2},\\
		R_{r+s}&=P_{s-1}R_{r-1}+P_{s}R_{r-2}+P_{s-2}R_{r-3}.
	\end{align*}
\end{lemma}
\noindent
We continue with the 2-adic valuation of Padovan and Perrin numbers, which was fully characterized by Irmak \cite[Lemmas 2.5 and 2.6]{NI}.
\begin{lemma}\label{L2}
	For $n\ge 0$, we have 
	\begin{align*}
		\nu_2\left(P_{n}\right)= 
		\begin{cases}
			0, & \mbox{if } n\equiv 0,1,2,5,\pmod{7};\\
			\nu_2(n+4)+1, & \mbox{if } n\equiv 3\pmod{7};\\
			\nu_2((n+3)(n+17))+1, & \mbox{if }  n\equiv 4\pmod{7};\\
			\nu_2((n+1)(n+8))+1, & \mbox{if }  n\equiv 6\pmod{7}.
		\end{cases}
	\end{align*}
\end{lemma}
\begin{lemma}\label{L3}
	For $n\ge 0$, we have
	$$
	\nu_2\left(R_{n}\right)= 
	\begin{cases}
		0, & \mbox{if } n\equiv 0,3,5,6,\pmod{7};\\
		1, & \mbox{if }  n\equiv 2,4,11\pmod{14};\\
		2, & \mbox{if }  n\equiv 9\pmod{14};\\
		\nu_2(n-1)+1, & \mbox{if } n\equiv 1\pmod{7}.
	\end{cases}
	$$
\end{lemma}
\noindent
We also quote two results on these sequences found by Bravo and Irmak \cite[Lemmas 3.5 and 3.6]{BrIrmak}.
\begin{lemma}\label{L4}
	For the integers $j$ and $k,t\ge 1$, we get
	$$
	P_{7\cdot 2^{t}k+j}\equiv\begin{cases}
		P_{j}\pmod{2^{t+2}}, & \text{if } j\equiv 1,2,4\pmod{7},\\
		P_{j}+k\cdot 2^{t+1}\pmod{2^{t+2}}, & \text{otherwise}.
	\end{cases}
	$$
\end{lemma}

\begin{lemma}\label{L5}
	For the integers $j$ and $k,t\ge 1$, we have
	$$
	R_{7\cdot 2^{t}k+j}\equiv\begin{cases}
		R_{j}\pmod{2^{t+2}}, & \text{if } j\equiv 0,2,6\pmod{7},\\
		R_{j}+k\cdot 2^{t+1}\pmod{2^{t+2}}, & \text{otherwise}.
	\end{cases}
	$$
\end{lemma}

The following lemma provides doubling recurrences for the 
Perrin numbers at indices $7\cdot 2^k$ and $-7\cdot 2^k$, 
which will be essential in the computational reduction step.

\begin{lemma}\label{lem:doubling}
	The following hold.
	\begin{itemize}
		\item[\rm(i)] $R_7 = 7$, $R_{14} = 51$, $R_{28} = 2627$, 
		$R_{-7} = -1$, $R_{-14} = -13$, $R_{-28} = 67$.
		\item[\rm(ii)] For all $k \ge 2$,
		\begin{equation*}
			R_{7\cdot 2^{k+1}} = R_{7\cdot 2^k}^2 - 2R_{-7\cdot 2^k}.
		\end{equation*}
		\item[\rm(iii)] For all $k \ge 2$,
		\begin{equation*}
			R_{-7\cdot 2^{k+1}} = R_{-7\cdot 2^k}^2 - 2R_{7\cdot 2^k}.
		\end{equation*}
	\end{itemize}
\end{lemma}

\begin{proof}
	Part (i) follows by direct computation from Lemma \ref{eq:binet_perrin}. For part (ii), 
	we use Lemma \ref{eq:binet_perrin} to expand
	\begin{align*}
		R_{7\cdot 2^k}^2 
		&= \big(\alpha^{7\cdot 2^k} + \beta^{7\cdot 2^k} 
		+ \gamma^{7\cdot 2^k}\big)^2\\
		&= \alpha^{7\cdot 2^{k+1}} + \beta^{7\cdot 2^{k+1}} 
		+ \gamma^{7\cdot 2^{k+1}}
		 + 2\big[(\alpha\beta)^{7\cdot 2^k} 
		+ (\alpha\gamma)^{7\cdot 2^k} 
		+ (\beta\gamma)^{7\cdot 2^k}\big].
	\end{align*}
	Since $\alpha\beta\gamma = 1$ by Vieta's formulas applied 
	to $\phi(X)$, we have 
	$\alpha\beta = \gamma^{-1}$, $\alpha\gamma = \beta^{-1}$, 
	and $\beta\gamma = \alpha^{-1}$. Therefore,
	\begin{equation*}
		(\alpha\beta)^{7\cdot 2^k} + (\alpha\gamma)^{7\cdot 2^k} 
		+ (\beta\gamma)^{7\cdot 2^k} 
		= \alpha^{-7\cdot 2^k} + \beta^{-7\cdot 2^k} 
		+ \gamma^{-7\cdot 2^k} = R_{-7\cdot 2^k}.
	\end{equation*}
	Substituting back gives
	\begin{equation*}
		R_{7\cdot 2^k}^2 = R_{7\cdot 2^{k+1}} + 2R_{-7\cdot 2^k},
	\end{equation*}
	from which part (ii) follows immediately. 
	For part (iii), we expand $R_{-7\cdot 2^k}^2$ using Lemma \ref{eq:binet_perrin} as
	\begin{align*}
		R_{-7\cdot 2^k}^2 
		&= \big(\alpha^{-7\cdot 2^k} + \beta^{-7\cdot 2^k} 
		+ \gamma^{-7\cdot 2^k}\big)^2\\
		&= \alpha^{-7\cdot 2^{k+1}} + \beta^{-7\cdot 2^{k+1}} 
		+ \gamma^{-7\cdot 2^{k+1}} + 2\big[(\alpha\beta)^{-7\cdot 2^k} 
		+ (\alpha\gamma)^{-7\cdot 2^k} 
		+ (\beta\gamma)^{-7\cdot 2^k}\big].
	\end{align*}
	The first three terms give $R_{-7\cdot 2^{k+1}}$. For the 
	remaining terms, since $\alpha\beta\gamma=1$, we have 
	$(\alpha\beta)^{-1}=\gamma$, $(\alpha\gamma)^{-1}=\beta$, 
	and $(\beta\gamma)^{-1}=\alpha$. Therefore,
	\begin{equation*}
		(\alpha\beta)^{-7\cdot 2^k} + (\alpha\gamma)^{-7\cdot 2^k} 
		+ (\beta\gamma)^{-7\cdot 2^k} 
		= \gamma^{7\cdot 2^k} + \beta^{7\cdot 2^k} 
		+ \alpha^{7\cdot 2^k} = R_{7\cdot 2^k}.
	\end{equation*}
	Substituting back gives
	\begin{equation*}
		R_{-7\cdot 2^k}^2 = R_{-7\cdot 2^{k+1}} + 2R_{7\cdot 2^k},
	\end{equation*}
	from which part (iii) follows immediately.
\end{proof}

Further, the following lemma provides index-shifting recurrences for any sequence satisfying the Padovan recurrence, and will be applied to both $\{P_n\}_{n\ge 0}$ and $\{R_n\}_{n\ge 0}$ in the reduction steps of the proofs of Theorems \ref{thm:1} and \ref{thm:2}.
\begin{lemma}\label{lem:shift}
	Let $\{U_n\}_{n\in\mathbb{Z}}$ be any sequence satisfying 
	the recurrence $U_{n+3} = U_{n+1} + U_n$, and let $j$ be 
	any integer. Then for all $k \ge 1$,
	\begin{itemize}
		\item[\rm(i)]
		\begin{equation*}
			U_{j-7\cdot 2^{k+1}} = -U_j R_{7\cdot 2^k} 
			+ U_{j+7\cdot 2^k} + U_{j-7\cdot 2^k}R_{-7\cdot 2^k},
		\end{equation*}
		\item[\rm(ii)]
		\begin{equation*}
			U_{j+7\cdot 2^{k+1}} = -U_j R_{-7\cdot 2^k} 
			+ U_{j-7\cdot 2^k} + U_{j+7\cdot 2^k}R_{7\cdot 2^k}.
		\end{equation*}
	\end{itemize}
\end{lemma}

\begin{proof}
	Let the Binet formula of $\{U_n\}_{n\in\mathbb{Z}}$ be
	\begin{equation*}
		U_n = a\alpha^n + b\beta^n + c\gamma^n,
	\end{equation*}
	for some $a, b, c \in \mathbb{Q}(\alpha,\beta)$.
	We prove part (i) by expanding $U_j R_{7\cdot 2^k}$ using both Binet formulas as follows:
	\begin{align*}
		U_j R_{7\cdot 2^k} 
		&= \left(a\alpha^j + b\beta^j + c\gamma^j\right)
		\big(\alpha^{7\cdot 2^k} + \beta^{7\cdot 2^k} 
		+ \gamma^{7\cdot 2^k}\big)\\
		&= \big(a\alpha^{j+7\cdot 2^k} + b\beta^{j+7\cdot 2^k} 
		+ c\gamma^{j+7\cdot 2^k}\big)
		 + a\alpha^{j-7\cdot 2^k}
		\big[(\alpha\beta)^{7\cdot 2^k} 
		+ (\alpha\gamma)^{7\cdot 2^k}\big]\\
		&\quad + b\beta^{j-7\cdot 2^k}
		\big[(\beta\alpha)^{7\cdot 2^k} 
		+ (\beta\gamma)^{7\cdot 2^k}\big]
		 + c\gamma^{j-7\cdot 2^k}
		\big[(\gamma\alpha)^{7\cdot 2^k} 
		+ (\gamma\beta)^{7\cdot 2^k}\big].
	\end{align*}
	Since $\alpha\beta\gamma = 1$, we have $(\alpha\beta)^{7\cdot 2^k} 
	= \gamma^{-7\cdot 2^k}$, $(\alpha\gamma)^{7\cdot 2^k} 
	= \beta^{-7\cdot 2^k}$, and $(\beta\gamma)^{7\cdot 2^k} 
	= \alpha^{-7\cdot 2^k}$. Substituting these relations and 
	regrouping, we obtain
	\begin{align*}
		U_j R_{7\cdot 2^k} 
		&= U_{j+7\cdot 2^k} 
		+ \big(a\alpha^{j-7\cdot 2^k} 
		+ b\beta^{j-7\cdot 2^k} 
		+ c\gamma^{j-7\cdot 2^k}\big)
		\big(\alpha^{-7\cdot 2^k} + \beta^{-7\cdot 2^k} 
		+ \gamma^{-7\cdot 2^k}\big)\\
		&\quad - \big(a\alpha^{j-7\cdot 2^{k+1}} 
		+ b\beta^{j-7\cdot 2^{k+1}} 
		+ c\gamma^{j-7\cdot 2^{k+1}}\big)\\
		&= U_{j+7\cdot 2^k} + U_{j-7\cdot 2^k}R_{-7\cdot 2^k} 
		- U_{j-7\cdot 2^{k+1}},
	\end{align*}
	from which part (i) follows immediately. For part (ii), we expand $U_j R_{-7\cdot 2^k}$ using 
	both Binet formulas as follows:
	\begin{align*}
		U_j R_{-7\cdot 2^k} 
		&= \left(a\alpha^j + b\beta^j + c\gamma^j\right)
		\big(\alpha^{-7\cdot 2^k} + \beta^{-7\cdot 2^k} 
		+ \gamma^{-7\cdot 2^k}\big)\\
		&= \big(a\alpha^{j-7\cdot 2^k} + b\beta^{j-7\cdot 2^k} 
		+ c\gamma^{j-7\cdot 2^k}\big)
		 + a\alpha^{j+7\cdot 2^k}
		\big[(\alpha\beta)^{-7\cdot 2^k} 
		+ (\alpha\gamma)^{-7\cdot 2^k}\big]\\
		&\quad + b\beta^{j+7\cdot 2^k}
		\big[(\beta\alpha)^{-7\cdot 2^k} 
		+ (\beta\gamma)^{-7\cdot 2^k}\big]
		 + c\gamma^{j+7\cdot 2^k}
		\big[(\gamma\alpha)^{-7\cdot 2^k} 
		+ (\gamma\beta)^{-7\cdot 2^k}\big].
	\end{align*}
	Since $\alpha\beta\gamma = 1$, we have 
	$(\alpha\beta)^{-7\cdot 2^k} = \gamma^{7\cdot 2^k}$, 
	$(\alpha\gamma)^{-7\cdot 2^k} = \beta^{7\cdot 2^k}$, 
	and $(\beta\gamma)^{-7\cdot 2^k} = \alpha^{7\cdot 2^k}$. 
	Substituting these relations and regrouping, we obtain
	\begin{align*}
		U_j R_{-7\cdot 2^k} 
		&= U_{j-7\cdot 2^k}
		 + \big(a\alpha^{j+7\cdot 2^k} 
		+ b\beta^{j+7\cdot 2^k} 
		+ c\gamma^{j+7\cdot 2^k}\big)
		\big(\alpha^{7\cdot 2^k} + \beta^{7\cdot 2^k} 
		+ \gamma^{7\cdot 2^k}\big)\\
		&\quad - \big(a\alpha^{j+7\cdot 2^{k+1}} 
		+ b\beta^{j+7\cdot 2^{k+1}} 
		+ c\gamma^{j+7\cdot 2^{k+1}}\big)\\
		&= U_{j-7\cdot 2^k} + U_{j+7\cdot 2^k}R_{7\cdot 2^k} 
		- U_{j+7\cdot 2^{k+1}},
	\end{align*}
	from which part (ii) follows immediately.
\end{proof}

\subsection{The $2$-adic valuation of $\{P_n+1\}_{n\ge 0}$ and $\{R_n-1\}_{n\ge 0}$}\label{sec:2adic}

In this section, we establish explicit formulas for the $2$-adic valuations of $\{P_n+1\}_{n\ge 0}$ and $\{R_n-1\}_{n\ge 0}$ except for a residue class of $n$ modulo $112$ for $\{P_n+1\}_{n\ge 0}$ and two residue classes of $n$ modulo $14$ for $\{R_n-1\}_{n\ge 0}$. For these exceptional classes, we establish a bound on the corresponding $2$-adic valuations of the above numbers for $n<2.3\times 10^{15}$, a bound that will be obtained later by applying Baker's theory to our equations. These are the fundamental tools used in the proofs of Theorems~\ref{thm:1} and~\ref{thm:2} to reduce the large bounds on $m$ and $n$ obtained from Baker's theory to constants small enough to allow for computational verification. Let's start with the shifted Padovan sequence.

\begin{proposition}\label{prop:Pn}
	For $n\ge 0$ such that $n\not\equiv 50\pmod{112}$, we have
	\begin{equation*}
		\nu_{2}(P_{n}+1)= 
		\begin{cases}
			0, & \text{if } n\equiv 3,4,6\pmod{7},\\
		\nu_{2}(n+9)+1, & \text{if } n\equiv 5\pmod{7},\\
			1, & \text{if } n\equiv 0,1.2,7,9\pmod{14},\\
			3, & \text{if } n\equiv 8\pmod{28},\\
			5, & \text{if } n\equiv 22\pmod{56},\\
			\nu_{2}(n+6)+4, & \text{if } n\equiv 106\pmod{112}.
		\end{cases}
	\end{equation*}
	Moreover, if $n\equiv 50\pmod{112}$ and $n<2.3\times 10^{15}$, 
	then $\nu_{2}(P_{n}+1)\le 56$.
\end{proposition}

\begin{proof}
	First, we verify that the seven listed residue classes are 
	mutually exclusive and exhaustive. Every integer $n\ge 0$ 
	satisfies exactly one of the relations $n\equiv 0,1,2,3,4,5,6\pmod{7}$. 
	The class $n\equiv 1\pmod{7}$ is further refined by the 
	remainder of $n$ modulo higher powers of $2$ by $7$, 
	which yields the subcases modulo $14$, $28$, $56$, and $112$. 
	It can be directly verified that the union of the seven classes 
	covers all non-negative integers exactly once.
	
	Now let's look at each case separately.
	
	\medskip
	\noindent\textbf{Case 1.} $n\equiv 3,4,6\pmod{7}$.
	
	By Lemma \ref{L2}, $\nu_2(P_n)\ge 1$ for $n$ in each of 
	these residue classes, so $P_n$ is even. Therefore, $P_n+1$ 
	is odd, which implies that $\nu_2(P_n+1)=0$.
	
		\medskip
	\noindent\textbf{Case 2.} $n\equiv 5\pmod{7}$.
	
	Let $n=7s+5$ for some integer $s\ge 0$. Then, $n+9=7(s+2)$. By the fundamental theorem of arithmetic, $s+2$ can be uniquely expressed as $2^{t}k$ for some integer $t\ge 0$ and some odd $k$. Therefore, every $n$ in this residue 
	class can be uniquely expressed as $n=7\cdot 2^t k-9$ for 
	some $t\ge 0$ and some odd $k$. Since 
	$j=-9\equiv 5\pmod{7}$, it follows from Lemma \ref{L4} that
	\begin{equation*}
		P_{7\cdot 2^t k-9}\equiv P_{-9}+k\cdot 2^{t+1}
		\pmod{2^{t+2}}.
	\end{equation*}
	It follows from Lemma \ref{eq:binet_padovan} that $P_{-9}=c_{\alpha}\alpha^{-9}+c_{\beta}\beta^{-9}+c_{\gamma}\gamma^{-9}=-1$. Therefore,
	\begin{equation*}
		P_n+1\equiv k\cdot 2^{t+1}\pmod{2^{t+2}}.
	\end{equation*}
	Since $k$ is odd, it follows that $\nu_2(P_n+1)=t+1$. 
	Finally, since $n+9=7\cdot 2^t k$ and $k$ is odd, we 
	have that $\nu_2(n+9)=t$, so $\nu_2(P_n+1)=\nu_2(n+9)+1$, as stated.
	
	\medskip
	\noindent\textbf{Case 3.} $n\equiv 0,1,2,7,9\pmod{14}$.
	
	We prove by weak induction that $P_n\equiv 1\pmod{4}$ for all such $n$, which implies that $P_{n}+1=4t+2$ for some integer $t\ge 0$ and so $\nu_2(P_n+1)=\nu_{2}(2(2t+1))=1$.
	
	\smallskip
	\textit{Subcase} $n\equiv 0,7\pmod{14}$. Let $n=7k$ for some integer 
	$k\ge 0$. The base case holds, since $P_0=1\equiv 1\pmod{4}$. 
	Suppose that $P_{7k}\equiv 1\pmod{4}$ for some fixed integer
	$k\ge 0$. Then, applying Lemma \ref{L1} with $r=7k+2$ and $s=5$, we obtain
	\begin{align*}
		P_{7(k+1)} = P_{(7k+2)+5} 
		&= P_{4}P_{7k+1} + P_{3}P_{7k+2} 
		+ P_{2}P_{7k}\\
		&= 2P_{7k+1}+2P_{7k+2}+P_{7k}\\
		&\equiv 2P_{7k+4}+1\pmod{4},
	\end{align*}
	where we used the fact that $P_2=1$ and $P_3=P_4=2$, and the recurrence relation of $\{P_{n}\}_{n}$. Since $7k+4\equiv 4\pmod{7}$, 
	Lemma~\ref{L2} states that $P_{7k+4}$ is even. Hence, $2P_{7k+4}\equiv 0\pmod{4}$. Therefore,
	\begin{equation*}
		P_{7(k+1)}\equiv 1\pmod{4}.
	\end{equation*}
	
	\smallskip
	\textit{Subcase} $n\equiv 2,9\pmod{14}$. Let $n=7k+2$ for some integer 
	$k\ge 0$. Since $P_2=1\equiv 1\pmod{4}$, the case base holds. 
	Suppose that $P_{7k+2}\equiv 1\pmod{4}$ for some fixed integer $k\ge 0$. Then, 
	applying Lemma~\ref{L1} with $r=7k+2$ and $s=7$, and 
	using the fact that $P_4=2$, $P_5=3$, and $P_6=4$, we obtain
	\begin{align*}
		P_{7(k+1)+2} = P_{(7k+2)+7} 
		&= P_{6}P_{7k+1}+P_{5}P_{7k+2}
		+P_{4}P_{7k}\\
		&= 4P_{7k+1}+3P_{7k+2}+2P_{7k}\\
		&\equiv 3+2P_{7k}\pmod{4}.
	\end{align*}
	Since $7k\equiv 0\pmod{7}$, Lemma~\ref{L2} states that $P_{7k}$ 
	is odd, so $2P_{7k}\equiv 2\pmod{4}$. Therefore,
	\begin{equation*}
		P_{7(k+1)+2}\equiv 1\pmod{4}.
	\end{equation*}
	
	\smallskip
	\textit{Subcase} $n\equiv 1\pmod{14}$. Let $n=14k+1$ for some integer $k\ge 0$. Since $P_1=1\equiv 1\pmod{4}$, the case base is satisfied. Suppose that $P_{14k+1}\equiv 1\pmod{4}$ for some fixed integer $k\ge 0$. Applying 
	Lemma~\ref{L1} with $r=14k+2$ and $s=13$, and using the fact that 
	$P_{10}=12$, $P_{11}=16$, and $P_{12}=21$, we obtain
	\begin{align*}
		P_{14(k+1)+1} = P_{(14k+2)+13}
		&= P_{12}P_{14k+1}+P_{11}P_{14k+2}
		+P_{10}P_{14k}\\
		&= 21P_{14k+1}+16P_{14k+2}+12P_{14k}\\
		&\equiv P_{14k+1}\equiv 1\pmod{4}.
	\end{align*}
	
	\medskip
	\noindent\textbf{Case 4.} $n\equiv 8\pmod{28}$.
	
	We will prove by weak induction that $P_n\equiv 7\pmod{16}$, which gives $P_{n}+1=16t+8$ for some integer $t\ge 0$, so 
	$\nu_2(P_n+1)=\nu_{2}(8(2t+1))=3$. Let $n=28k+8$ for an integer $k\ge 0$. Since $P_8=7\equiv 7\pmod{16}$, the case base holds. Suppose that $P_{28k+8}\equiv 7\pmod{16}$ for some fixed $k\ge 0$. Applying 
	Lemma~\ref{L1} with $r=28k+10$ and $s=26$, and using the fact that
	$P_{23}=465\equiv 1\pmod{16}$, $P_{24}=616\equiv 8\pmod{16}$, and $P_{25}=816\equiv 0\pmod{16}$, 
	we obtain
	\begin{align*}
		P_{28(k+1)+8} = P_{(28k+10)+26}
		&= P_{25}P_{28k+9}+P_{24}P_{28k+10}
		+P_{23}P_{28k+8}\\
		&\equiv 8P_{28k+10}+7\pmod{16}.
	\end{align*}
	Since $28k+10\equiv 3\pmod{7}$, Lemma~\ref{L2} gives 
	$\nu_2(P_{28k+10})=\nu_2(28k+14)+1=\nu_2(14(2k+1))+1=2$. Therefore, $P_{28k+10}=4q$ for 
	some odd $q$, and 
	$8P_{28k+10}=32q\equiv 0\pmod{16}$. Thus,
	\begin{equation*}
		P_{28(k+1)+8}\equiv 7\pmod{16}.
	\end{equation*}
	
	\medskip
	\noindent\textbf{Case 5.} $n\equiv 22\pmod{56}$.
	
	We will prove by induction that $P_n\equiv 31\pmod{64}$, so
	$\nu_2(P_n+1)=\nu_{2}(64t+32)=\nu_{2}(32(2t+1))=5$, where $t$ is a non-negative integer. Let $n=56k+22$ for an integer $k\ge 0$. Since $P_{22}=351\equiv 31\pmod{64}$,  the base case holds. Suppose that 
	$P_{56k+22}\equiv 31\pmod{64}$ for some fixed $k\ge 0$. Applying Lemma~\ref{L1} 
	with $r=56k+24$ and $s=54$, and using
	$P_{51}=1221537\equiv 33\pmod{64}$, $P_{52}=1618192\equiv 16\pmod{64}$, and $P_{53}=2143648\equiv 32\pmod{64}$, we obtain
	\begin{align*}
		P_{56(k+1)+22} = P_{(56k+24)+54}
		&= P_{53}P_{56k+23}+P_{52}P_{56k+24}
		+P_{51}P_{56k+22}\\
		&\equiv 32P_{56k+23}+16P_{56k+24}
		-1\pmod{64}.
	\end{align*}
Since $56k+23\equiv 2\pmod{7}$, Lemma~\ref{L2} gives that $P_{56k+23}$ is odd. Thus, $32P_{56k+23}\equiv 32\pmod{64}$. Since $56k+24\equiv 3\pmod{7}$, Lemma~\ref{L2} gives $\nu_2(P_{56k+24})=\nu_2(56k+28)+1=\nu_2(28(2k+1))+1=3$. Thus, $P_{56k+24}=8w$ for some odd $w$, and $16P_{56k+24}\equiv 0\pmod{64}$. Therefore,
	\begin{equation*}
		P_{56(k+1)+22}\equiv 31\pmod{64}.
	\end{equation*}
	
	\medskip
	\noindent\textbf{Case 6.} $n\equiv 106\pmod{112}$.
	
Let $n=112s+106$ for some integer $s\ge 0$. Then, $n+6=112(s+1)$. If $s\ge 1$, then, by the fundamental theorem of arithmetic, $s+1$ can be written uniquely as $2^{t}k$ for some integer $t\ge 0$ and some odd $k$. Thus, every $n$ in this class of residues can be uniquely expressed as $n=7\cdot 2^{t+4}k-6$ for some $t\ge 0$ and some odd $k$. First, from Lemma \ref{eq:binet_padovan} and the fact that $P_{-6}=-1$ (by Lemma \ref{eq:binet_padovan}), we have the identity
	\begin{equation}\label{eq:key-case7}
		P_n+1 = P_n - P_{-6} 
		= c_{\alpha}\alpha^{-6}(\alpha^{n+6}-1)
		+c_{\beta}\beta^{-6}(\beta^{n+6}-1)
		+c_{\gamma}\gamma^{-6}(\gamma^{n+6}-1).
	\end{equation}
	We now determine the 2-adic structure of 
	$\alpha^{n+6}-1=\alpha^{7\cdot 2^{t+4}k}-1$ 
	by iteratively squaring. For any polynomial 
	$Q(X)\in\mathbb{Z}[X]$, we compute $Q(\alpha)$ 
	by taking the remainder of the division by the 
	minimal polynomial $\phi(X)$. In this way, 
	we obtain the following chain of congruences:
	\begin{eqnarray}
	\label{eq:alpha}
		\alpha^7 
		&=& 2\alpha^2+2\alpha+1
		= 1+2(\alpha^{2}+\alpha),\nonumber\\
		\alpha^{7\cdot 2} 
		&=& 1+2^2(3\alpha^2+4\alpha+2),\\
		\alpha^{7\cdot 2^2} 
		&=& 1+2^3(77\alpha^2+102\alpha+58),\nonumber\\
		\alpha^{7\cdot 2^3} 
		&=& 1+2^4(101137\alpha^2+133978\alpha+76346),\nonumber\\
		\alpha^{7\cdot 2^4} 
		&=& 1+2^5 f(\alpha),\nonumber
	\end{eqnarray}
	where
	\begin{equation*}
		f(\alpha) := 348972965593\alpha^2
		+462290754114\alpha+263431897850.
	\end{equation*}
	Each step involves squaring the previous expression 
	and reducing it modulo $\phi(\alpha)=0$.
	
	Now, by raising $\alpha^{7\cdot 2^4}=1+2^5 f(\alpha)$ 
	to the power $2^tk$ using the binomial theorem, we obtain
	\begin{equation*}
		\alpha^{7\cdot 2^{t+4}k} 
		= \sum_{s=0}^{2^t k}\binom{2^t k}{s} 2^{5s}f(\alpha)^s.
	\end{equation*}
	For $s\ge 2$, Kummer's theorem gives 
	$\nu_2\binom{2^t}{s}=t-\nu_2(s)$, and since 
	$\nu_2(s)\le \log s/\log 2\le 5s-9$ for $s\ge 2$ 
	(since the function $h(s):=5s-9-\log s/\log 2$ vanishes 
	at $s=2$ and has a positive derivative $5-1/(s\log 2)>0$ 
	for all $s\ge 2$), 
	we obtain
	\begin{equation*}
		\nu_2\left(\binom{2^t}{s}2^{5s}\right) 
		\ge t-\nu_2(s)+5s \ge t+9\qquad {\text{\rm for}}\qquad s\ge 2.
	\end{equation*}
	Thus, 
	$$\alpha^{7\cdot 2^{t+4}}\equiv 1+2^{t+5}f(\alpha)\pmod {2^{t+9}}.
	$$
Since $k$ is odd, raising to the power $2^t k$ instead of $2^t$ yields the same congruence modulo $2^{t+9}$, since the additional factor of $k$ merely multiplies the leading term by the odd integer $k$ and leaves the higher-order terms still divisible by $2^{t+9}$. Therefore,
	\begin{equation}\label{eq:alpha-congruence}
		\alpha^{7\cdot 2^{t+4}k} 
		\equiv 1+2^{t+5}kf(\alpha)\pmod{2^{t+9}},
	\end{equation}
	and similarly for $\beta$ and $\gamma$. Substituting 
	\eqref{eq:alpha-congruence} into \eqref{eq:key-case7} 
	and using Lemma \ref{eq:binet_padovan} gives
	$$
		P_n+1 
		\equiv 2^{t+5}k\bigl(c_{\alpha}\alpha^{-6}f(\alpha)
		+c_{\beta}\beta^{-6}f(\beta)
		+c_{\gamma}\gamma^{-6}f(\gamma)\bigr)\equiv 2^{t+5}k\,M\pmod{2^{t+9}},
	$$
	where
	$$
		M := 348972965593P_{-4}
		+462290754114P_{-5}
		+263431897850P_{-6}.
	$$
	Here, we used the identity 
	$$c_{\alpha}\alpha^{-6}f(\alpha)+c_{\beta}\beta^{-6}f(\beta)
	+c_{\gamma}\gamma^{-6}f(\gamma)=348972965593P_{-4}+462290754114P_{-5}	+263431897850P_{-6},$$
 which follows from Lemma \ref{eq:binet_padovan}. Since $k$ is odd and 
 $$
 \nu_2(M)=\nu_2(198858856264)=\nu_{2}(2^3\cdot 1229\cdot 20225677)=3,
 $$ 
 we conclude that
	\begin{equation*}
		\nu_2(P_n+1) = t+5+\nu_2(k)+\nu_2(M) 
		= t+8.
	\end{equation*}
	Since $\nu_2(n+6)=\nu_2(7\cdot 2^{t+4}k)=t+4$, this gives
	\begin{equation*}
		\nu_2(P_n+1) = \nu_2(n+6)+4.
	\end{equation*}
	
	\medskip
	\noindent\textbf{An exceptional case.} $n\equiv 50\pmod{112}$.
	
	Unlike the previous cases, the residue class $n\equiv 50\pmod{112}$ does not admit 
	a uniform closed-form expression for $\nu_2(P_n+1)$. 
	Instead, we determine the maximum possible value of 
	$\nu_2(P_n+1)$ over all $n$ in this class with 
	$n<2.3\times 10^{15}$ using the following algorithm, which 
	is essentially based on Lemmas~\ref{lem:doubling} 
	and~\ref{lem:shift}.
	
	\medskip
	\noindent\textit{Algorithm.} Set $n_0=50$ and note that $\nu_2(P_{50}+1)=\nu_2(922112)=\nu_{2}(2^{9}\cdot 1801)=9$. At each step, given a current index 
	$n_r$, we compute $\nu_2(P_{n_r+7\cdot 2^{r+1}}+1)$ 
	using Lemma~\ref{lem:shift} applied with 
	$U_n=P_n$ modulo $2^{100}$, after first computing 
	$R_{7\cdot 2^k}\pmod{2^{100}}$ and 
	$R_{-7\cdot 2^k}\pmod{2^{100}}$ for 
	$k=0,1,\ldots,100$ using the doubling recurrence relations from 
	Lemma \ref{lem:doubling}. If adding $7\cdot 2^{r+1}$ 
	to the current index increases the value of $\nu_2(P_n+1)$, we set 
	$n_{r+1}=n_r+7\cdot 2^{r+1}$; otherwise we try 
	$7\cdot 2^{r+2}$, and so on. The algorithm terminates 
	when the next step would require an index greater than 
	$2.3\times 10^{15}$, which occurs approximately when
	$r=49$, since $7\cdot 2^{49}>2.3\times 10^{15}$. 
	The maximum value of $\nu_2(P_n+1)$ produced by this 
	algorithm is the constant $k_1$.
	
An implementation of this algorithm in SageMath yields 
\begin{equation*}
	k_{1}=\nu_{2}(P_{919832845308882}+1)= 56.
\end{equation*}
See Appendix \ref{app1} for more information. This completes the proof of Proposition \ref{prop:Pn}.
\end{proof}
\begin{remark}\label{rem:2adic-k}
The algorithm described in Case 8 of the proof of Proposition \ref{prop:Pn} also generates an explicit approximation of a non-rational $2$-adic zero $n_0$ of $\{P_n+1\}_{n\ge 0}$ with $n\equiv 50\pmod{112}$. Writing $n_0=50+7\cdot 2^{4}k$ where $k\in\mathbb{Z}_2$ is a $2$-adic integer, the algorithm computes $k$ modulo $2^{100}$ by recording which steps $7\cdot 2^r$ are accepted (contributing $2^{r-4}$ to $k$) and which are skipped. This gives
	\begin{equation*}
		k \equiv 1024997913023633590925131382
		\pmod{2^{100}},
	\end{equation*}
whose $2$-adic expansion (digits listed from least significant) is
$$
{{ 0 1 1 0 1 1 1 0 0 1 1 1 1 1 1 1 1 1 1 1 1 0 0 1 0 0 0 0 1 1 0 0 0 0 0 1 1 1 1 0 1 1 1 0 0 0 0 0 1 0 1 0 0 1 1 0 0 0 1 0 1 1 0 0 0 0 0 0 0 0 1 1 1 1 0 1 1 0 1 1 1 1 1 1 0 0 1 0 1 1_2.}}
$$
The corresponding $2$-adic zero satisfies
	\begin{equation*}
		n_0 \equiv 
		114799766258646962183614714834
		\pmod{7\cdot 2^{100}}.
	\end{equation*}
It can be verified by calculation that the number $n_0$ from the previous residue class satisfies $P_{n_0}+1\equiv 0\pmod{2^{100}}$, which confirms that $n_0$ is, in fact, a $2$-adic zero of the sequence $\{P_n+1\}_{n\ge 0}$.
\end{remark}

We will now move on to the shifted Perrin sequence.

\begin{proposition}\label{prop:Rn}
	For $n\ge 0$ such that $n\not\equiv 5,10\pmod{14}$, we have
	\begin{equation*}
		\nu_{2}(R_{n}-1)= 
		\begin{cases}
			0, & \text{if } n\equiv 1,2,4\pmod{7},\\
			1, & \text{if } n\equiv 0,3,7,13\pmod{14},\\
			2, & \text{if } n\equiv 6\pmod{14},\\
			\nu_{2}(n+2)+1, & \text{if } n\equiv 12\pmod{14}.
		\end{cases}
	\end{equation*}
	Moreover, if $n\equiv 5,10\pmod{14}$ and 
	$n<2.2\times 10^{15}$, then $\nu_{2}(R_{n}-1)\le 51$.
\end{proposition}

\begin{proof}
	We verify that the listed residue classes, 
	together with the exceptional classes 
	$n\equiv 5,10\pmod{14}$, are mutually exclusive and 
	cover all non-negative integers. Every $n\ge 0$ belongs 
	to exactly one class modulo $7$. The class 
	$n\equiv 1\pmod{7}$ (which contains both 
	$n\equiv 1\pmod{14}$ and $n\equiv 8\pmod{14}$) and 
	$n\equiv 5\pmod{7}$ (which contains 
	$n\equiv 5\pmod{14}$ and $n\equiv 12\pmod{14}$) are 
	refined by the residue modulo $14$, and it can be 
	directly verified that all the classes together cover each 
	non-negative integer exactly once.
	
	We handle each case individually.
	
	\medskip
	\noindent\textbf{Case 1.} $n\equiv 1,2,4\pmod{7}$.
	
	By Lemma~\ref{L3}, $\nu_2(R_n)\ge 1$ for $n$ in each 
	of these residue classes, so $R_n$ is even. Therefore, 
	$R_n-1$ is odd, which gives $\nu_2(R_n-1)=0$.
	
	\medskip
	\noindent\textbf{Case 2.} $n\equiv 0,3,7,13\pmod{14}$.
	
	We assert that $R_n\equiv 3\pmod{4}$ for all such $n$, which 
	implies that $R_{n}=4t+3$ for some integer $t\ge 0$, and so $\nu_2(R_n-1)=\nu_{2}(2(2t+1))=1$. 
	
	\smallskip
	\textit{Subcase} $n\equiv 0,7\pmod{14}$. Let $n=7k$ for some integer $k\ge 0$. 
	The base case is satisfied since $R_0=3\equiv 3\pmod{4}$. Suppose that $R_{7k}\equiv 3\pmod{4}$ for some fixed 
	$k\ge 0$. Applying Lemma~\ref{L1} with $r=7k+3$ and 
	$s=4$, and using the fact that $P_2=1$ and $P_3=P_4=2$, we obtain
	\begin{align*}
		R_{7(k+1)} = R_{(7k+3)+4} = P_3 R_{7k+2}+P_4 R_{7k+1}+P_2 R_{7k}= 2R_{7k+2}+2R_{7k+1}+R_{7k}.
	\end{align*}
	Since $7k+1\equiv 1\pmod{7}$ and $7k+2\equiv 2\pmod{7}$, 
	Lemma~\ref{L3} implies $\nu_2(R_{7k+1})\ge 1$ and 
	$\nu_2(R_{7k+2})\ge 1$, so both $R_{7k+1}$ and 
	$R_{7k+2}$ are even. Therefore, $2R_{7k+1}\equiv 2R_{7k+2} \equiv 0
	\pmod{4}$. Together 
	with the induction hypothesis, 
	we conclude
	\begin{equation*}
		R_{7(k+1)}\equiv 3\pmod{4},
	\end{equation*}
	completing the induction.
	
	\smallskip
	\textit{Subcase} $n\equiv 3 \pmod {14}$.
	Let $n=14k+3$ for some integer $k\ge 0$. By the congruence \eqref{eq:alpha}, we have 
	$$
	\alpha^{14k+3}=(\alpha^{14} )^k \alpha^3=[1+2^2(3\alpha^2+4\alpha+2)]^k \alpha^3\equiv \alpha^3\pmod {4}.
	$$
	If we apply the above by substituting $\alpha$ with $\beta$ and $\gamma$ and using Lemma \ref{eq:binet_perrin}, we obtain
	$$
	R_{14k+3}-1=\alpha^{14k+3}+\beta^{14k+3}+\gamma^{14k+3}-1\equiv \alpha^3+\beta^3+\gamma^3-1\equiv R_3-1\equiv 2\pmod 4,
	$$
	which shows that $R_{n}\equiv 3\pmod{4}$ in this case. 
	
	\smallskip
	\textit{Subcase} $n\equiv 13\pmod{14}$. Let $n=14k+13$ for some integer $k\ge 0$. 
	The base case holds since $R_{13}=39\equiv 3\pmod{4}$. 
	Suppose that $R_{14k+13}\equiv 3\pmod{4}$ for some fixed $k\ge 0$. 
	Applying Lemma~\ref{L1} with $r=14k+14$ and $s=13$, 
	and using the fact that $P_{11}=16$, $P_{12}=21$, and $P_{13}=28$, 
	we conclude
	\begin{align*}
		R_{14(k+1)+13} = R_{(14k+14)+13}&= P_{12}R_{14k+13}+P_{13}R_{14k+12}+P_{11}R_{14k+11}\\
		&= 21R_{14k+13}+28R_{14k+12}+16R_{14k+11}\\
		&\equiv R_{14k+13}\pmod{4},
	\end{align*}
	completing the induction.
	
	\medskip
	\noindent\textbf{Case 3.} $n\equiv 6\pmod{14}$.
	
	We prove by induction that $R_n\equiv 5\pmod{8}$, which gives $R_{n}=8t+5$ for some integer $t\ge 0$, and so 
	$\nu_2(R_n-1)=\nu_2(4(2t+1))=2$. Let $n=14k+6$ for some integer $k\ge 0$. 
	The base case is satisfied since $R_6=5\equiv 5\pmod{8}$. Suppose that $R_{14k+6}\equiv 5\pmod{8}$ for some fixed $k\ge 0$. Applying 
	Lemma~\ref{L1} with $r=14k+7$, $s=13$, and using the fact that $P_{11}=16$, 
	$P_{12}=21$, and $P_{13}=28$, we obtain
	\begin{align*}
		R_{14(k+1)+6} = R_{(14k+7)+13}
		&= P_{12}R_{14k+6}+P_{13}R_{14k+5}
		+P_{11}R_{14k+4}\\
		&= 21R_{14k+6}+28R_{14k+5}+16R_{14k+4}\\
		&\equiv 1+4R_{14k+5}\pmod{8}.
	\end{align*}
    Now $14k+5\equiv 5\pmod{7}$, so 
	$\nu_2(R_{14k+5})=0$ by Lemma~\ref{L3}, meaning 
	$R_{14k+5}=2q+1$ for some integer $q\ge 0$. Thus,
	$4R_{14k+5}=8q+4\equiv 4\pmod{8}$.
	Therefore,
	$$
	R_{14(k+1)+6}\equiv 5\pmod{8}.
	$$
	
	\medskip
	\noindent\textbf{Case 4.} $n\equiv 12\pmod{14}$.
	
	Let $n=14s+12$ for some integer $s\ge 0$. Then, $n+2=7\cdot 2(s+1)$. If $s\ge 1$, then, by the fundamental theorem of arithmetic, we have that $s+1$ can be written uniquely as $2^{t-1}k$ for some integer $t\ge 1$ and some odd $k$. Therefore, every $n$ of this kind 
	can be written uniquely as $n=7\cdot 2^t k-2$ for some $t\ge 1$ and some odd $k$.
	Since $n\equiv -2\pmod{7}$, we apply 
	Lemma~\ref{L5} with $j=-2$ to obtain
	\begin{equation*}
		R_{7\cdot 2^t k-2}\equiv R_{-2}+k\cdot 2^{t+1}
		\pmod{2^{t+2}}.
	\end{equation*}
	Since $R_{-2}=\alpha^{-2}+\beta^{-2}+\gamma^{-2}=1$ by Lemma \ref{eq:binet_perrin}, we have
	\begin{equation*}
		R_n-1\equiv k\cdot 2^{t+1}\pmod{2^{t+2}}.
	\end{equation*}
	Since $k$ is odd, $\nu_2(R_n-1)=t+1$. Since 
	$n+2=7\cdot 2^t k$ with $k$ odd, $\nu_2(n+2)=t$, 
	which leads to $\nu_2(R_n-1)=t+1=\nu_2(n+2)+1$.
	
	\medskip
	\noindent\textbf{Exceptional cases.} 
	$n\equiv 5,10\pmod{14}$.
	
For these two residual classes, the $2$-adic valuation of $R_n-1$ does not admit a uniform closed-form expression. We determine the constant $k_2$ by applying the same greedy algorithm as in the exceptional case of Proposition~\ref{prop:Pn}, now with $U_n=R_n$, starting from each base value $n_0\in\{5,10\}$ in turn, and running until the index exceeds $2.2\times 10^{15}$. In each case, $R_{7\cdot 2^k}\pmod{2^{100}}$ and $R_{-7\cdot 2^k}\pmod{2^{100}}$ are precomputed using Lemma~\ref{lem:doubling}, and the index shift uses Lemma~\ref{lem:shift} with $U_n=R_n$. An implementation in SageMath yields $k_2=\nu_{2}(R_{1085137218429451}-1)=49$ when $n\equiv 5\pmod{14}$ and $k_{2}=\nu_{2}(R_{683114118940578}-1)=51$ when $n\equiv 10\pmod{14}$. 

This ends the proof of Proposition \ref{prop:Rn}.
\end{proof}

\begin{remark}\label{rem:2adic-k-perrin}
Similarly, the algorithm generates explicit approximations of the non-rational $2$-adic zeros of $\{R_n-1\}_{n}$ in the exceptional classes. 
	
For $n\equiv 5\pmod{14}$, writing $n_0=5+7\cdot 2k$ with $k\in\mathbb{Z}_2$, the algorithm yields
	\begin{equation*}
		k \equiv 53607125039940024673042731045
		\pmod{2^{100}},
	\end{equation*}
with $2$-adic expansion beginning 
	$(1010010000001001001101000001010\ldots)_2$,
and the corresponding approximation to the $2$-adic zero satisfies
	\begin{equation*}
		n_0 \equiv 750499750559160345422598234635
		\pmod{7\cdot 2^{100}}.
	\end{equation*}
	
For $n\equiv 10\pmod{14}$, writing $n_0=10+7\cdot 2k$ with $k\in\mathbb{Z}_2$, the algorithm yields
	\begin{equation*}
		k \equiv 22778962594365194611691311828
		\pmod{2^{100}},
\end{equation*}
with $2$-adic expansion beginning $(0010101101001010111000001010110\ldots)_2$, and the corresponding approximation of the $2$-adic zero satisfies
	\begin{equation*}
		n_0 \equiv 318905476321112724563678365602
		\pmod{7\cdot 2^{ 100}}.
	\end{equation*}
In each case, it is computationally verified that $R_{n_0}-1\equiv 0\pmod{2^{99}}$, which confirms that the algorithm has converged to within $99$ $2$-adic digits of a genuine zero of $\{R_n-1\}_{n}$.
\end{remark}

\begin{remark}\label{rem:2adic-zeros}
The formulas in Propositions \ref{prop:Pn} and \ref{prop:Rn} have a natural explanation via the theory of 2-adic zeros of linear recurrence sequences developed in \cite{BiluLuca2025}. By \cite[Theorem 1.1]{BiluLuca2025}, if $\{U_n\}_{n\in\mathbb{Z}}$ is a linear recurrence sequence with a genuine integer zero at $n=a$, then $\nu_2(U_n)$ grows like $\nu_2(n-a)$ in a 2-adic neighborhood of $a$.
	
For $U_n = P_n+1$: the genuine integer zeros are $n=-6$ and $n=-9$, as stated in \cite[Theorem~2.3]{Bravo2023}. These integers are the only ones with $P_n=-1$, which explains the terms $\nu_2(n+6)+4$ and $\nu_2(n+9)+1$ in Proposition \ref{prop:Pn}.
	
For $U_n = R_n-1$: the only genuine integer zero is $n=-2$, as stated in \cite[Theorem~2.2]{Bravo2023}. This integer is the only one with $R_{n}=1$, which explains the term $\nu_2(n+2)+1$ in Proposition \ref{prop:Rn}.
	
The exceptional residue class $n\equiv 50\pmod{112}$ in Proposition \ref{prop:Pn} and the classes $n\equiv 5,10\pmod{14}$ in Proposition \ref{prop:Rn} correspond to \emph{non-rational} 2-adic zeros of $\{P_n+1\}_{n\in \mathbb{Z}}$ and $\{R_n-1\}_{n\in \mathbb{Z}}$, respectively. Such zeros exist according to general theory but they are not rational numbers; approximations to their 2-adic digits can be computed algorithmically, as described in Remarks \ref{rem:2adic-k} and \ref{rem:2adic-k-perrin}.
\end{remark}

For greater thoroughness, the complementary sequences studied by Bravo and Irmak \cite{BrIrmak} admit a similar analysis. The shifted Padovan sequence $\{P_n-1\}_{n\in \mathbb{Z}}$ has zeros at $n=-26,-13,-12,-7,-5,-2,0,1,2$; while the shifted Perrin sequence $\{R_n+1\}_{n\in \mathbb{Z}}$ has zeros at $n=-29,-11,-7,-1$.

\subsection{Linear forms in three logarithms}

Our problems can be reduced to linear forms in three logarithms of algebraic numbers with very small absolute value (exponentially small in the coefficients of the linear form). Therefore, lower bounds for these linear forms exceeding upper bounds, and in which all the constants involved are explicit, reduce our problems to a finite amount of computation. We will use the special case of three logarithms in Matveev's theorem. Later, we present a slight simplification of Theorem 2.1 from \cite{M} applied with $n=3$, which can be found in \cite{MV}. First, we define the absolute logarithm Weil height of an algebraic number.

\begin{definition}\label{def}
	Let $ \eta $ be an algebraic number of degree $d$ over $\mathbb{Q}$ with minimal primitive polynomial 
	$$ a_{d}X^{d}+a_{d-1}X^{d-1}+\cdots +a_{0}=a_{d}(X-\eta^{(1)})\cdots (X-\eta^{(d)})\in \mathbb{Z}\left[X\right]. $$ 
	The absolute logarithmic Weil height of $\eta$ is given by 
	$$h(\eta)= \dfrac{1}{d}\Bigg(\log (|a_d|) + \sum_{j=1}^{d}\max \big\{\log \big|\eta^{(j)}\big|,0\big\}\Bigg).$$
\end{definition}

\begin{theorem}\label{Matveev}
Let $\alpha_{1}$, $\alpha_{2}$ and $\alpha_{3}$ be three distinct non-zero algebraic numbers, let $\log \alpha_{1}$, $\log \alpha_{2}$ and $\log \alpha_{3}$ be $\mathbb{Q}-$linearly independent logarithms of these algebraic numbers and let $b_{1}$, $b_{2}$ and $b_{3}$ be rational integers with $b_1\neq 0$. Put
$$
\Lambda=b_{1}\log \alpha_{1}+b_{2}\log \alpha_{2}+b_{3}\log \alpha_{3}.
$$
Let
$$
D=[\mathbb{Q}(\alpha_1,\alpha_2,\alpha_3):\mathbb{Q}]\quad \text{and}\quad \chi =[\mathbb{R}(\alpha_{1},\alpha_{2},\alpha_{3}):\mathbb{R}].
$$
Let $A_1$, $A_2$ and $A_3$ be positive real numbers, which satisfy
$$
A_{j}\ge \max \{D\,h(\alpha_{j}),|\log \alpha_{j}|\}\quad (1\le j\le 3),
$$
where $h$ is the absolute logarithmic Weil height.

Assume that
$$
B\ge \max \{|b_{j}|A_{j}/A_{1}:1\le j\le 3\}.
$$
Also define
$$
\kappa_1:=\frac{5\cdot 16^{5}}{6\chi}e^{3}(7+2\chi)\left(\frac{3e}{2}\right)^{\chi}\left(26.25+\log (D^{2}\log (eD))\right).
$$
If $\Lambda \neq 0$, then
$$
\log |\Lambda|>-\kappa_{1}D^{2}A_{1}A_{2}A_{3}\log (1.5eDB\log (eD)).
$$
\end{theorem}

All computational verifications in this work were performed using SageMath 10.6.

\section{Proof of Theorem \ref{thm:1}}

In this section, we solve the Diophantine equation
\begin{equation}\label{main:eqn1}
P_n=W_m
\end{equation}
in integers $n\ge 0$ and $m\ge 1$. If $m\le 3$, then $P_{0}=P_{1}=P_{2}=1=W_{1}$ and $P_{8}=7=W_{2}$ are the only Woodall numbers in the Padovan sequence. Suppose that $(m,n)$ is a solution of equation \eqref{main:eqn1} with $m\ge 4$. Then $P_{n}=W_{m}\ge 4\cdot 2^{4}-1=63$. This implies that $n\ge 16$ since $P_{15}=49<63<65=P_{16}$.

On the other hand, we know that $4^{2}\le 2^{4}$. Suppose that $k^{2}\le 2^{k}$ for some integer $k\ge 4$. Then
$$
\frac{(k+1)^{2}}{k^{2}}=\left(1+\frac{1}{k}\right)^{2}\le \left(1+\frac{1}{4}\right)^{2}=1.5625<2.
$$
Thus, $(k+1)^{2}<2k^{2}\le 2\cdot 2^{k}=2^{k+1}$. By mathematical induction, we can conclude that $m^{2}\le 2^{m}$ for all $m\ge 4$. It follows that $W_{m}+1\le 2^{3m/2}$ for all $m\ge 4$. Furthermore, $P_{n}+1>\alpha^{n-2}$ for all $n\ge 0$ by Lemma \ref{lem:growth_P}. Therefore, from equation \eqref{main:eqn1} we obtain $\alpha^{n-2}<2^{3m/2}$. Hence,
\begin{equation}\label{E6}
n<4m+2.
\end{equation}

\subsection{Absolute upper bound for $m$}
From the above, it suffices to find an absolute upper bound for $m$ to show that equation \eqref{main:eqn1} has a finite number of solutions. Our result is as follows.

\begin{lemma}\label{lem:P1}
	If equation \eqref{main:eqn1} has a solution with $m\ge 4$ and $n\ge 16$, then
	\begin{align*}
		m <  5.6\times 10^{14}.
	\end{align*}
\end{lemma}
\begin{proof}
Using Lemma \ref{eq:binet_padovan}, we can rewrite \eqref{main:eqn1} as
$$
c_{\alpha}\alpha^{n}-m\cdot 2^{m}=-1-(c_{\beta}\beta^{n}+c_{\gamma}\gamma^{n}).
$$
Dividing the above equation by $m\cdot 2^{m}$ and taking absolute value we get
\begin{align*}
\left|2^{-m}\cdot \left(\frac{m}{c_{\alpha}}\right)^{-1}\cdot \alpha^{n}-1\right|=\left|\frac{-1-(c_{\beta}\beta^{n}+c_{\gamma}\gamma^{n})}{m\cdot 2^{m}}\right|< \frac{1+0.5\cdot 0.87^{16}}{4\cdot 2^{m}}<\frac{0.3}{2^{m}},
\end{align*}
where we used the fact that $m\ge 4$, $|c_{\beta}|=|c_{\gamma}|<0.25$, $|\beta|=|\gamma|<0.87$, and $n\ge 16$. Since $0.3/2^{m}<1/2$ for all $m\ge 1$, we have that
\begin{equation*}
|m\log 2+\log (m/c_{\alpha})-n\log \alpha |<2^{-m}.
\end{equation*}
Thus
\begin{equation}\label{E5}
\log |m\log 2+\log (m/c_{\alpha})-n\log \alpha|<-m\log 2.
\end{equation}
Put $\alpha_{1}=2$, $\alpha_{2}=m/c_{\alpha}$, $\alpha_{3}=\alpha$, $b_{1}=m$, $b_{2}=1$, and $b_{3}=-n$. Then 
$$
D=[\mathbb{Q}(2,m/c_{\alpha},\alpha):\mathbb{Q}]=[\mathbb{Q}(\alpha):\mathbb{Q}]=3\qquad {\text{\rm  and}}\qquad \chi =[\mathbb{R}(2,m/c_{\alpha},\alpha):\mathbb{R}]=[\mathbb{R}:\mathbb{R}]=1.
$$ 
On the other hand, it follows from Definition \ref{def} that $h(2)=\log 2$, $h(m)=\log m$, $h(c_{\alpha})=(\log 23)/3$ and $h(\alpha)=(\log \alpha)/3$, since their minimal primitive polynomials are $X-2$, $X-m$, $23X^3 - 23X^2 + 6X -1$ and $\phi(X)$, respectively. Since $h(m/c_{\alpha})\le h(m)+h(c_{\alpha})$, it follows that the real numbers 
$$
A_{1}:=3\log 2,\quad A_{2}:=3\log m +\log 23,\quad {\text{\rm and}} \quad A_{3}:=\log \alpha
$$ 
satisfy 
$$
A_{1}\ge \max \{3h(2),|\log 2|\},\quad A_{2}\ge \max \{3h(m/c_{\alpha}),|\log (m/c_{\alpha})|\},\quad {\text{\rm  and}}\quad A_{3}\ge \max \{3h(\alpha),|\log \alpha|\}.
$$
To choose $B$, note that it follows from \eqref{E6} that
$$
\max \{m,(3\log m +\log 23)/3\log 2,n(\log \alpha/3\log 2)\}=m.
$$
Therefore, we choose $B=m$. We have
$$
\kappa_1=\frac{5\cdot 16^{5}}{6}e^{3}(7+2)\left(\frac{3e}{2}\right)\left(26.25+\log (9\log (3e))\right).
$$
Observe that $\Lambda :=m\log 2+\log (m/c_{\alpha})-n\log \alpha \neq 0$. To see this, we consider the $\mathbb{Q}$-automorphism $\sigma$ of the Galois extension $\mathbb{K}:=\mathbb{Q}(\alpha,\beta)$ over $\mathbb{Q}$ defined by $\sigma(\alpha)=\beta$, $\sigma(\beta)=\alpha$ and $\sigma(\gamma)=\gamma$. If $\Lambda =0$ then $\sigma(\Lambda)=0$ and we get
$$
c_{\beta}\beta^{n}=\sigma(c_{\alpha} \alpha^{n})=m\cdot 2^{m}.
$$
Thus,
$$
1=1\cdot 1^{n}>|c_{\beta}||\beta|^{n}=m\cdot 2^{m}\ge 4\cdot 2^{4}=64.
$$
Hence, $\Lambda \neq 0$. By Theorem \ref{Matveev}, we conclude that
\begin{equation}\label{E3}
\log |\Lambda|>-27\kappa_{1}(\log 2)(3\log m+\log 23)(\log \alpha)\log (4.5em\log (3e)).
\end{equation}
Then from \eqref{E5} and \eqref{E3} we get that
\begin{equation}\label{E2}
m<27\kappa_{1}(\log \alpha)(3\log m+\log 23)\log (4.5em\log (3e)).
\end{equation}
Let
$$
f(m)=m-27\kappa_{1}(\log \alpha)(3\log m+\log 23)\log (4.5em\log (3e)).
$$
We want to find where $f(m)=0$. Note that
$$
f^{\prime}(m)=1-\frac{27\kappa_{1}(\log \alpha) [3\log m+3\log (4.5em\log (3e)) + \log 23]}{m},
$$
so the Newton-Raphson iteration is 
$$
m_{k+1}=m_{k}-\frac{f(m_{k})}{f^{\prime}(m_{k})}\quad (k\ge 0).
$$
An initial value of $m_{0}=10^{15}$ was chosen. The first four iterations are shown below.
\begin{table}[H]
\centering
\begin{tabular}{|c|c|c|}
\hline
$k$ & $m_{k}$ & $f(m_k)$ \\ \hline
0 & $10^{15}$ & $4.24287967246031 \times 10^{14}$ \\ \hline
1 & $5.61950326138484 \times 10^{14}$ & $4.20233296312069 \times 10^{12}$ \\ \hline
2 & $5.57503287195068 \times 10^{14}$ & $9.46359543562500 \times 10^{8}$ \\ \hline
3 & $5.57502285275635 \times 10^{14}$ & $48.625$ \\ \hline
4 & $5.57502285275583 \times 10^{14}$ & $\approx 0$ \\ \hline
\end{tabular}
\caption{Convergence of the Newton-Raphson method for $m_{0}=10^{15}$.}
\end{table}
Therefore, \eqref{E2} holds if 
$$
m<5.6\times 10^{14}.
$$
\end{proof}

\subsection{Lowering the upper bound of $m$}

We now use Proposition \ref{prop:Pn} to reduce the previous absolute upper bound for $m$ to a constant small enough to allow a direct computational search for the solutions of equation \eqref{main:eqn1}. The key observation is as follows. By applying the $2$-adic 
valuation to equation \eqref{main:eqn1}, we obtain
\begin{align*}
	\nu_2(P_n + 1) = \nu_2(m\cdot 2^m) 
	= m + \nu_2(m) \ge m,
\end{align*}
so that
\begin{equation}\label{eq:m-bound-nu}
	m \le \nu_2(P_n+1).
\end{equation}
We divide the argument into two cases, depending on whether 
$n$ belongs to the class of exceptional residues from 
Proposition~\ref{prop:Pn} or not.

\medskip
\noindent\textbf{(i) Non-exceptional cases:} 
$n\not\equiv 50\pmod{112}$.

By Proposition~\ref{prop:Pn}, the $2$-adic valuation 
$\nu_2(P_n+1)$ is explicitly given in terms of $n$. 
In all non-exceptional cases, the largest possible 
value is either a fixed constant or of the form 
$\nu_2(n+\ell)+c$ for small constants $\ell$ and $c$. 
Since $\nu_2(q) \le \log q/\log 2$ for any positive 
integer $q$, we obtain in all non-exceptional cases
that
\begin{equation}\label{eq:nu-log}
	\nu_2(P_n+1) \le 
	\max\left\{5,\, \nu_2(n+9)+1,\, 
	\nu_2(n+6)+4\right\} 
	\le \frac{\log(n+6)}{\log 2}+4.
\end{equation}
Combining \eqref{eq:m-bound-nu} and \eqref{eq:nu-log} 
with \eqref{E6}, 
we obtain
\begin{equation*}
	m < \frac{\log(4m+8)}{\log 2}+4.
\end{equation*}
Since the right-hand side is concave down as a
function of real $m>1$, a 
routine calculation shows that this inequality implies that
\begin{equation*}
	m \le 9.
\end{equation*}

\medskip
\noindent\textbf{(ii) Exceptional case:} 
$n\equiv 50\pmod{112}$.

By \eqref{E6} and Lemma~\ref{lem:P1}, we have that $n <4(5.6\times 10^{14})+2< 2.3\times 10^{15}$. 
Therefore, Proposition~\ref{prop:Pn} applies and yields 
$\nu_2(P_n+1)\le 56$. Combining this with \eqref{eq:m-bound-nu}, we obtain
\begin{equation*}
	m \le 56.
\end{equation*}

Therefore, in all cases, $m \le 56$. Then, $P_{n}\le W_{56}$. Hence, $n\le 153$ since $W_{56}\in (P_{153},P_{154})$. Finally, we write a simple code in SageMath to determine all solutions $(m,n)$ of equation \eqref{main:eqn1} with $m\in [4,56]$ and $n\in [16,153]$. No solution were found. This concludes the proof of Theorem \ref{thm:1}.

\section{Proof of Theorem \ref{thm:2}}

This proof follows the same general strategy as the proof of Theorem \ref{thm:1}. Now, we solve the Diophantine equation
\begin{equation}\label{main:eqn2}
R_n= C_m
\end{equation}
in integers $m\ge 1$ and $n\ge 0$. For $m\le 3$, the only Cullen number in the Perrin sequence is 
$$
R_{0}=R_{3}=3=C_1.
$$ 
Suppose that $(m,n)$ is a solution of equation \eqref{main:eqn2} with $m\ge 4$. Then, $R_n=C_{m}\ge C_{4}=65$. This implies that $n\ge 15$ since $65\in (R_{14},R_{15})$. 

On the other hand, given Lemma \ref{lem:growth_R} and the fact that $m\le 2^{m/2}$ for all $m\ge 4$, we obtain
\begin{equation}\label{eq:n-bound}
n<4m+1.
\end{equation}

\subsection{Absolute upper bound for $m$}
From \eqref{eq:n-bound}, it again suffices to impose an absolute bound on $m$ to conclude that equation \eqref{main:eqn2} has a finite number of solutions. Our result is as follows.

\begin{lemma}\label{lem:R1}
	If equation \eqref{main:eqn2} has a solution with $m\ge 4$ and $n\ge 15$, then
	\begin{align*}
		m <   5.4\times 10^{14}.
	\end{align*}
\end{lemma}
\begin{proof}
Using Lemma \ref{eq:binet_perrin}, we can rewrite \eqref{main:eqn2} as
$$
2^{-m}\cdot m^{-1}\cdot \alpha^{n}-1=\frac{1-(\beta^{n}+\gamma^{n})}{m\cdot 2^{m}}.
$$
Taking the absolute value in the above equality and using the fact that $|\beta|=|\gamma|<0.87$, $m\ge 4$, and $n\ge 15$, we obtain
$$
|2^{-m}\cdot m^{-1}\cdot \alpha^{n}-1|<\frac{0.3}{2^{m}}.
$$ 
Since $0.3/2^{m}<1/2$ for $m\ge 1$, we obtain
\begin{equation*}
|m\log 2+\log m-n\log \alpha|<2^{-m}.
\end{equation*}
Therefore
\begin{equation}\label{E8}
\log |m\log 2+\log m-n\log \alpha|<-m\log 2.
\end{equation}
Now we use Theorem \ref{Matveev} with $\alpha_{1}=2$, $\alpha_{2}=b_{1}=m$, $\alpha_{3}=\alpha$, $b_{2}=1$, and $b_{3}=-n$. If 
$$
\Lambda :=m\log 2+\log m-n\log \alpha=0,
$$ 
then conjugating this relation by the automorphism $\sigma$ and taking the absolute value, we get $|\beta|^{n}=m\cdot 2^{m}$, which is not possible for $m\ge 1$ since $|\beta|<1$. Therefore, $\Lambda \neq 0$. 

In this application of Theorem \ref{Matveev}, we can take, just as in the previous one, 
$$
D=3,\quad \chi =1,\quad A_{1}=3\log 2,\quad {\text{\rm  and}}\quad A_{3}=\log \alpha.
$$ 
Therefore, $\kappa_{1}$ remains unchanged as well. This time, it suffices to set $A_{2}:=3\log m$ to ensure that $A_{2}\ge \max \{3h(m),|\log m|\}$. From \eqref{eq:n-bound} we can conclude that $\max \{m,\log m/\log 2,n(\log \alpha/3\log 2)\}=m$, so $B=m$ will suffice. Then
\begin{equation}\label{E9}
\log |m\log 2+\log m-n\log \alpha|>-81\kappa_{1}(\log 2)(\log m)(\log \alpha)\log (4.5em\log (3e)).
\end{equation}
From \eqref{E8} and \eqref{E9} we conclude that
\begin{equation}\label{E4}
m<81\kappa_{1}(\log \alpha)(\log m)\log (4.5em\log (3e)).
\end{equation}
When we apply Newton's method to the function $g(m):= m-81\kappa_{1}(\log \alpha)(\log m)\log (4.5em\log (3e))$ with $m_{0}=10^{14}$, we observe in the fifth iteration that $g(m)=0$ when $m\approx 5.3987565\times 10^{14}$, so \eqref{E4} holds if
$$
m<5.4\times 10^{14}.
$$
\end{proof}

\subsection{Lowering the upper bound of $m$}

Taking the $2$-adic valuation on both sides of \eqref{main:eqn2} written as $R_n - 1 = m\cdot 2^m$, we obtain
\begin{align*}
	\nu_2(R_n - 1) = \nu_2(m\cdot 2^m) 
	= m + \nu_2(m) \ge m,
\end{align*}
so that
\begin{equation}\label{E7}
m\le \nu_2(R_n - 1).
\end{equation}
We again split into two cases.

\medskip
\noindent\textbf{(i) Non-exceptional cases:} 
$n\not\equiv 5,10\pmod{14}$.

By Proposition~\ref{prop:Rn}, the largest possible 
value of $\nu_2(R_n-1)$ in non-exceptional cases is 
of the form $\nu_2(n+2)+1$. Since 
$\nu_2(q)\le \log q/\log 2$ for any positive integer 
$q$, we have
\begin{equation}\label{E1}
	\nu_2(R_n-1) \le \frac{\log(n+2)}{\log 2}+1.
\end{equation}
Combining \eqref{E7} and \eqref{E1} with \eqref{eq:n-bound}, we get
\begin{equation*}
	m<\frac{\log(4m+3)}{\log 2}+1.
\end{equation*}
Since the right-hand side is concave down as
a function of real $m>1$, a
routine computation shows that this inequality requires
$$
m\le 5.
$$

\medskip
\noindent\textbf{(ii) Exceptional cases:} 
$n\equiv 5,10\pmod{14}$.

By Lemma \ref{lem:R1} and \eqref{eq:n-bound} we have $n<2.2\times 10^{15}$, 
so Proposition~\ref{prop:Rn} applies and gives 
$\nu_2(R_n-1)\le 51$. Therefore, by \eqref{E7} we get
\begin{equation*}
	m \le 51.
\end{equation*}
We conclude, therefore, that $m\le 51$ in all cases. This implies that $R_{n}\le C_{51}$, which holds only if $n\le 139$, since $C_{51}\in (R_{139},R_{140})$. A search in SageMath reveals that equation \eqref{main:eqn2} has no solutions for $4\le m\le 51$ and $15\le n\le 139$. This completes the proof of Theorem \ref{thm:2}.

\section*{Acknowledgments} 
The first author thanks the Mathematics division of Stellenbosch University for funding his PhD studies. The third author was partially supported by the 2024 ERC Synergy Grant ``DynAMiCs".

\section*{Addresses}

$ ^{1} $ Mathematics Division, Stellenbosch University, Stellenbosch, South Africa.

Email: \url{hbatte91@gmail.com}

Email: \url{fluca@sun.ac.za}
\\
$ ^{2} $ Departamento de Matem\'aticas, Universidad del Cauca, Popay\'an, Colombia.

Email: \url{fbravo@unicauca.edu.co}
\\
$ ^{3} $ Max Planck Institute for Software Systems, Saarbr\"ucken, Germany.

\section*{Competing interests}
On behalf of all authors, the corresponding author states that there is no Conflict of interest.

\section*{Funding}
The first author was supported by a PhD scholarship from the Mathematics Division of Stellenbosch University.  The third author was partially supported by the 2024 ERC Synergy Grant ``DynAMiCs". 

\section*{Data Availability}
Data sharing is not applicable to this article as no datasets were generated or analyzed during the current study.
 \newpage
 \appendix
 \section{Appendices}
 \subsection{Py Code I}\label{app1}
 \begin{verbatim}
MOD = 2^100
R_mod = IntegerModRing(MOD)

# Step 1: Padovan companion matrix

M_pad = matrix(R_mod, [
  [0, 1, 0],
  [0, 0, 1],
  [1, 1, 0]
])
v0_pad = vector(R_mod, [1, 1, 1])

def padovan_mod_fast(n):
  """Compute P_n mod 2^100 for any n >= 0. Uses matrix exponentiation: O(log n) steps."""
  if n <= 2:
    return ZZ(1)
  vn = M_pad^n * v0_pad
  return ZZ(vn[0])

# Verify against known values
print("Verification of matrix exponentiation:")
print(f"P_0  = {padovan_mod_fast(0)}  (expected 1)")
print(f"P_8  = {padovan_mod_fast(8)}  (expected 7)")
print(f"P_22 = {padovan_mod_fast(22)} (expected 351)")
print(f"P_50 mod 2^100 = {padovan_mod_fast(50)}")
print(f"nu_2(P_50 + 1) = "
  f"{valuation(padovan_mod_fast(50) + 1, 2)}")
print()
# ------------------------------------------------------------
# Step 2: Precompute R_{7*2^k} and R_{-7*2^k} mod 2^100 using Lemma lem:doubling 

R_pos = {}  # R_{7*2^k} mod 2^100
R_neg = {}  # R_{-7*2^k} mod 2^100

R_pos[0] = Integer(7)
R_pos[1] = Integer(51)
R_pos[2] = Integer(2627)
R_neg[0] = Integer(-1)
R_neg[1] = Integer(-13)
R_neg[2] = Integer(67)

for k in range(2, 65):
  R_pos[k+1] = (R_pos[k]^2 - 2*R_neg[k]) % MOD
  R_neg[k+1] = (R_neg[k]^2 - 2*R_pos[k]) % MOD

print("R values precomputed.")
print()

# ------------------------------------------------------------
# Step 3: Run the greedy algorithm. At each step r, we try adding 7*2^r to n_current.
# We compute P_{n_candidate} mod 2^100 directly via matrix exponentiation.
LIMIT = 2.3 * 10^15
n_current     = 50
P_current     = padovan_mod_fast(50)
delta_current = valuation((P_current + 1) % MOD, 2)

print("="*60)
print(f"Starting greedy algorithm")
print(f"n_0 = {n_current}")
print(f"nu_2(P_50 + 1) = {delta_current}  (expected 9)")
print("="*60)

r = 5  # first step is 7*2^5 = 224
max_delta_seen = delta_current

while True:
  step  = 7 * 2^r
  n_candidate = n_current + step

  # Termination condition
  if n_candidate >= LIMIT:
    print(f"\nAlgorithm terminates at r={r}.")
    print(f"Next candidate would be n = {n_candidate}")
    print(f"which exceeds the limit 2.3*10^15.")
    break

  # Compute P_{n_candidate} mod 2^100
  P_candidate    = padovan_mod_fast(n_candidate)
  delta_candidate = valuation((P_candidate + 1) % MOD, 2)

  print(f"r={r}: n + 7*2^{r} = {n_candidate}")
  print(f"  nu_2(P_{n_candidate} + 1) = {delta_candidate}")

  if delta_candidate > delta_current:
    n_current     = n_candidate
    P_current     = P_candidate
    delta_current = delta_candidate
    max_delta_seen = max(max_delta_seen, delta_current)
    print(f"  --> ACCEPTED.")
    print(f"      New n     = {n_current}")
    print(f"      New nu_2  = {delta_current}")
    r += 1
  else:
    print(f"  --> SKIPPED (nu_2 did not increase).")
    r += 1 

# ------------------------------------------------------------
# Step 4: Report
print()
print("="*60)
print("FINAL RESULT")
print("="*60)
print(f"Optimal n            = {n_current}")
print(f"c_1 = nu_2(P_n + 1)  = {delta_current}")
print(f"c_4 = 5*c_1 + 2      = {4*delta_current + 2}")
print()
print("These values go into the paper as:")
print(f"  m <= c_1 = {delta_current}")
print(f"  n <= c_4 = {4*delta_current + 2}")
 \end{verbatim}
\end{document}